\documentclass[12pt,a4paper,reqno]{amsart}
\usepackage{latexsym}
\usepackage{amsfonts}
\usepackage[mathscr]{eucal}
\usepackage{epsfig}
\usepackage{a4wide}
\usepackage{ifthen}
\usepackage{graphicx}

\newboolean{draft}
\setboolean{draft}{false}

\newcommand{\trunc}[1]{ {\lfloor #1 \rfloor} }

\def\@copyright{}
\newtheorem{theorem}{Theorem}[section]

\newtheorem{lemma}{Lemma}[section]

\newtheorem{corollary}{Corollary}[section]
\newtheorem{remark}{Remark}[section]

\newtheorem{example}{Example}[section]

\newcommand{\N}{ \mathbb{N} }

\newcommand{\R}{ \mathbb{R} }

\newcommand{\calI}{\mathcal{I}}

\newcommand{\calU}{\mathcal{U}}

\newcommand{\calZ}{\mathcal{Z}}

\newcommand{\eins}{{\mathbf 1}}

\newcommand{\Cov}{{\mbox{Cov\,}}}

\renewcommand{\circ}{\cdot}

\begin{document}

\bibliographystyle{apalike}

\pagestyle{plain}
\begin{titlepage}
\pagenumbering{arabic}
\title[MONITORING TIME SERIES]{}
\date{Version: October, $30$th, 2006}
\end{titlepage}

\maketitle

\begin{center}
  \Large
  MONITORING PROCEDURES TO DETECT UNIT ROOTS AND STATIONARITY
\end{center}
\vskip 1cm
\begin{center}
  Ansgar Steland\footnote{Address of correspondence: Prof. Dr. A. Steland, Institute of Statistics, RWTH Aachen University, W\"ullnerstr. 3, D-52056 Aachen, Germany.}\\
  Institute of Statistics\\
  RWTH Aachen University\\
  Germany\\
  steland@stochastik.rwth-aachen.de
\end{center}

\begin{abstract}
When analysing time series an important issue is to decide whether the time series
is stationary or a random walk. Relaxing these notions, we consider the problem to
decide in favor of the $ I(0) $- or $ I(1)$-property.
Fixed-sample statistical tests for that problem
are well studied in the literature. In this paper we provide first results for
the problem to monitor sequentially a time series. Our stopping times are based
on a sequential version of a kernel-weighted variance-ratio statistic.
The asymptotic distributions are established for $I(1)$ processes,
a rich class of stationary processes, possibly
affected by local nonparametric alternatives, and the local-to-unity model.
Further, we consider the two interesting change-point models where the time series
changes its behaviour after a certain fraction of the observations and derive the
associated limiting laws. Our Monte-Carlo studies show that the proposed
detection procedures have high power when interpreted as a hypothesis test,
and that the decision can often be made very early.
\\[1cm]
\end{abstract}

\newpage
\section*{Introduction}

For many time series, in particular for economic data, the question
whether the series is stationary or becomes stationary when taking
first order differences is a delicate problem. Fixed-sample tests
have been extensively studied in the statistics and econometrics
literature and this topic is still an active area of research. Most
proposed unit root tests are parametric approaches based on the
least squares estimator in an AR model. Under the random walk
hypothesis non-standard limiting distributions appear. Classic and
more recent references are Dickey and Fuller (1979), Rao (1978,
1980), Evans and Savin (1981), Chan and Wei (1987, 1988), and
Phillips and Perron (1988), Stock (1994a), Saikkonen and L\"utkepohl
(2003), and Lanne and Saikkonen (2003). Nonparametric tests have
been studied by Kwiatkowski et al. (1992), Bierens (1997), Breitung
(2002), and Giraitis et al. (2003). The KPSS test, proposed in the
first paper and also studied in detail in the latter two articles,
avoids a detailed specification of the process. It can be easily
used for testing both the null hypothesis of stationarity against
the unit root alternative, and vice versa, and, as shown by
simulations, is considerably more robust in terms of type I error
than most parametrically motivated tests. Thus we use that statistic
as a starting point to develop detection procedures which can be
used to detect a change from $ I(0) $ (covering stationarity) to $
I(1) $ (covering random walks), and vice versa. Similar detection
procedures related to the Dickey-Fuller statistic, which is often
more powerful but can be affected by severe size distortion, will be
studied by the author in a separate paper in detail (Steland, 2006).

As an example of a simple change-point model (regime switching model) capturing
this feature consider
$$
 Y_{n+1} = \phi_n Y_n + \epsilon_n \quad \mbox{where} \quad
 \left\{ \begin{array}{cc} \phi_n = 1, & n = 1, \dots, \trunc{N\vartheta}-1, \\
                        |\phi_n| < 1, & n = \trunc{N\vartheta}, \dots, N,
                        \end{array} \right.
$$
with mean-zero error terms $ \{ \epsilon_n \} $ and a change-point parameter
$ \vartheta \in (0,1) $. Until the change the {\em in-control model} of a $I(1)$ process holds.
This model is a special case of a $ I(1) $-to-$ I(0) $ change-point model
studied in this article. If $ \phi_n = \phi \in [-1,1] $
Lai and Siegmund (1983) studied fixed accuracy estimation of an AR parameter assuming
i.i.d. error terms by sampling until the Fisher information exceeds a constant.
Allowing for dependent errors,
we consider a different setup and study truncated stopping times of the type
$ S_N = \min \{ 1 \le n \le N : T_{N,n} < c \} $
for some {\em control statistic}, $ T_{N,n} $, and a {\em control limit} (critical value) $c$,
where monitoring stops latest at the $ N $th observation.
That maximum sample size, $N$, plays the role of a time horizon where
a decision is made in any case; if no signal is given,
the in-control model (null hypothesis) is accepted as a plausible model,
otherwise one stops concluding that
a change occurred and further measures may be in order. We stop latest at $N$, since often
the assumption implicit to many classic monitoring procedures with random sample size,
namely
that a process can and should be monitored forever, is unrealistic,
and approaches allowing to specify a time horizon may be more appropriate in many cases.
For example, consider financial portfolios.
Continuous or pseudo-continuous (daily) trading is often not feasible, due to cost
constraints and because identification of mid- and long-term investment
chances requires time- and cost-intensive analyses on a quarterly to yearly basis.
Between these analyses one should apply monitoring rules with time horizon to trigger additional updates, risk hedges, or other measures.
Having approximations to the distributions of the control statistic and the monitoring
rule (stopping time) for large $N$ are therefore of interest, thus motivating to
assume $ N \to \infty $ for asymptotic studies.

To allow the design of procedures satisfying arbitrary constraints, e.g. prespecified
type I error, average run length (ARL) or median run length,
we establish the limiting laws
which are functionals of Brownian motion or the Ornstein-Uhlenbeck process.
Besides the important change-point models mentioned above, we consider pure
random walks, stationary processes, local trend-stationary processes, and
the local-to-unity model, where a sequence of models is considered which converges to a random walk, as the maximum sample size, $N$, tends to $ \infty $.

Let us briefly comment on other related work.
The nonparametric detection of a change in the mean of a stationary time series based on
kernel-weighted averages and the problem of optimal kernel choice
has been studied by Steland (2004a, 2005a). For the related problem to detect a
change in the mean of a random walk see Steland (2005b).
A posteriori methods,
where observations after the change are also available, have been studied by
Ferger (1993, 1995), Hu\v{s}kov\'a (1999) and Hu\v{s}kov\'a and Slaby (2001).
For an approach based on jump-preserving statistics aiming at detecting
quickly large shifts see Pawlak, Rafaj\l owicz and Steland (2004) and Steland (2005c).
The problem to detect changes in a linear model has been recently studied
by Horv\'ath et al. (2004) using CUSUMs of residuals.

Retrospective change-point detection allowing for time series data has been studied
quite extensively by many authors.  Kr\"amer and Ploberger (1992)
study partial sums of OLS regression residuals to detect structural changes.
Bai (1994) established weak convergence of the sequential
empirical process of ARMA($p,q$)-residuals and constructed a CUSUM-type statistic
to detect a change in the distribution of the innovations.
Noting that, e.g., ML estimates usually can be written as arithmetic means
of stationary martingale differences, Lee et al. (2003) studied a CUSUM procedure
to detect changes in parametric time series models.
For work on structural breaks and changes in the trend function
in integrated variables, we refer to Nyblom (1989), Perron (1991),
Vogelsang (1997), Hansen and Johansen (1997), and Bai, Lumsdaine and Stock (1998).
For further references to the extensive literature about these issues we
refer to the references given in these papers.

The paper is organised as follows. In Section~\ref{Model} we explain the
proposed monitoring procedure and basic assumptions. Functional central limit
theorems (FCLTs) under general conditions are given in Section~\ref{MainResults}. Change-point problems for a change from $I(0)$ to $ I(1)$,
and vice versa, are discussed in Section~\ref{SecCP}.
Section~\ref{Simulations} provides Monte Carlo results to assess the accuracy and
performance of the considered stopping times
demonstrating that the procedure works very reliable and
often can detect stationarity earlier than a fixed sample test, and that
using a weighting scheme improves the detection of a change-point.

\section{Preliminaries, method, and assumptions}
\label{Model}

We will use the following nonparametric definitions of the notions $ I(0) $ and $ I(1) $.
A  time series $ \{ Y_n \} $ is called $ I(0) $, denoted by $ Y_n \sim I(0) $, if
\begin{equation}
\label{RWDefI0}
  N^{-1/2} \sum_{i=1}^{ \trunc{ Ns } } Y_i \Rightarrow \sigma B(s), \qquad s \in [0,1],
\end{equation}
holds for some constant $ 0 < \sigma < \infty $. Here and throughout the paper
$ \trunc{x} $ denotes the floor function, $B(s),\ s \in [0,1],$ denotes Brownian motion, and
$ \Rightarrow $ stands for weak convergence in the space $ D[0,1] $ of all
right-continuous functions with left-hand limits equipped with the Skorohod topology
given by the Skorohod metric $ d $. For that approach to weak convergence we refer
to Billingsley (1968) and Prigent (2003).
In terms of mixing and moment conditions, a sufficient condition for
(\ref{RWDefI0}) is, e.g., that $ \{ Y_n \} $ is a stationary $ \alpha $-mixing sequence
with $ E |Y_1|^{2+\delta} < \infty $ and $ \sum_k \alpha(k)^{2/(2+\delta)} < \infty $
for some $ \delta > 0 $ where $ \alpha(k) $ are the mixing coefficients
(e.g. Herrndorf (1985)). Some of our limit theorems
assume (\ref{RWDefI0}) under  additional weak regularity conditions.
We will formulate these conditions where needed.

$ \{ Y_n \} $ is integrated of order $1$, denoted by $ Y_n \sim  I(1) $, if
\begin{equation}
\label{RWYRW}
  N^{-1/2} Y_{ \trunc{Ns} } \Rightarrow \sigma B(s), \qquad s \in [0,1],
\end{equation}
as $ N \to \infty $,
and the differences, $ \Delta Y_n = Y_n - Y_{n-1} $, form a $ I(0) $ series.
Note that our definition of $ I(0) $ does not necessarily implies stationarity and
allows for a certain degree of dependence. The $ I(1) $ property is also quite general,
covering classic random walks $ Y_n = \sum_{i=1}^n u_i $ with mean-zero
i.i.d. innovations  $ \{ u_t \} $, but, e.g., also allows for random walks with
dependent innovations $ u_t $ satisfying a functional central limit theorem
of the type (\ref{RWDefI0}). However, long memory processes in the
sense that $ \sum_{k=-\infty}^\infty | \Cov(Y_1, Y_{1+k}) | = \infty $ are not allowed.

In the literature the $ I(0) $ property often means that the time series is
a linear process, $ \sum_{j=0}^\infty \psi_j Z_{t-j} $, where $ \{ Z_n \} $ is a weak
white noise sequence and the parameter sequence $ \{ \psi_j \} $ is absolutely
summable with $ \sum_j \psi_j \not= 0 $. However, our definitions have been used
by many other researchers, e.g., Stock (1994b),
and are appropriate to describe the classes of time series
which can be distinguished by the methods studied in this paper.

Let us now assume that the time series observations $ Y_1, \dots, Y_N $, $ N \in \N $, arrive sequentially at ordered time points $ t_1, \dots, t_N $. To simplify
presentation we assume $ t_n = n \in \N $, but more general time designs can be
handled as in Steland (2005b).
It is known that a robust nonparametric unit root test is given by considering
the ratio of the dispersion of the cumulated observations and the dispersion
of the observations, cf. Kwiatkowski et al. (1992) or Breitung (2002). Having in mind
change-point models where the time series changes its $I(0)$ respectively $I(1)$ property
at some unknown time point, we introduce appropriate kernel weights to avoid that
past observations dominate the statistic. We first introduce a
sequential kernel-weighted variance-ratio process which is appropriate to detect $ I(0) $
processes, and will then describe a modification to detect $I(1)$.
Define\footnote{
In a previous version of this paper
we scaled numerator and denominator by powers of $ N^{-1} $ instead
of $ \trunc{Ns}^{-1} $. Simulations indicate that both version have very similar power
properties. Scaling with $ \trunc{Ns}^{-1} $ has the advantage that the
values of the process needed to calculate the stopping time do not depend on
the maximum sample size $N$, but requires to put $ U_N(s) = 0 $ for $ s \in [0,1/N) $.
}
$ U_N(s) = 0 $ for $ s \in [0,1/N) $ and
\begin{equation}
\label{DefUN}
  U_N(s) =
    \frac{ \trunc{Ns}^{-3} \sum_{i=1}^{ \trunc{Ns} }
      \left( \sum_{j=1}^i Y_j \right)^2 K_h( i - \trunc{Ns} ) }
                { \trunc{Ns}^{-2} \sum_{j=1}^{ \trunc{Ns} } Y_j^2 },
  \ s \in [1/N,1].
\end{equation}
$ K_h( \circ ) = K( \circ / h ) / h $, where $K$ is a Lipschitz continuous
density function with mean $0$ and finite variance,
and $ h = h_N > 0 $ is a sequence of bandwidth parameters satisfying
$$
  N/h_N \to \zeta \in [1,\infty),
$$
as $ N \to \infty $.
The definition of the kernel weights, $ K_h( i - \trunc{Ns} ) $, requires
only a kernel function $K$ defined on $ (-\infty,0] $. Thus, we can and will assume that
$K$ is symmetric around $0$, otherwise put $ K(z) = K(-z) $, $ z > 0 $, if $ K $ is only
defined for $ z \le 0 $.
Clearly, $ U_N $ depends on the bandwidth parameter $h$. If $ K $ has support $ [-1,1] $,
$ U_N $ is a function of the current and the most recent $h$ observations.
However, our results allow for kernels with unbounded support, e.g., the Gaussian kernel.
To apply the procedure, one chooses the time horizon $N$ and the bandwidth $h$,
puts $ \zeta = N/h $, and uses the asymptotic distributional results given
in the subsequent sections as approximations.

Although technically not required, one usually employs kernels $ K(z) $ which are
decreasing in $ |z| $ and satisfy $ \lim_{|z| \to \infty} K(z) = 0 $, to ensure that
past partial sums have smaller weights than more recent ones.
The technical role of the denominator is to estimate a nuisance parameter summarising the influence
of the dependence structure of the time series on the asymptotic distribution of
the numerator of $ U_N $ if $ \{ Y_n \} $ is $ I(1) $.

If $ \{ Y_n \} $ is $ I(0) $, the numerator of $ U_N $ has a different convergence
rate, and one should also modify the denominator of $ U_N $.
Following Kwiatkowski et al. (1992) and others, let
\begin{equation}
\label{DefUNT}
  \widetilde{U}_N(s) = N^{-1} \sum_{i=1}^{ \trunc{Ns} } \left( \sum_{j=1}^i Y_j \right)^2 K_h( i - \trunc{Ns} ) \biggl/ s^2_{Nm}(s),
  \qquad s \in [0,1],
\end{equation}
where
$$
  s^2_{Nm}(s) = \frac{1}{N} \sum_{i=1}^{\trunc{Ns}} Y_i^2
    + 2 \sum_{k=1}^m w(k,m) \frac{1}{N} \sum_{i=1}^{\trunc{Ns}} Y_i Y_{i+k},
  \qquad s \in [0,1],
$$
is the process version of the Newey-West HAC estimator. $ w(k,m) $ is a
weighting function. One may use the Bartlett window, $ w(k,m) = 1 - k/m $,
as in Newey and West (1987) which guarantees nonnegativity of $ s^2_{Nm}(s) $.
For consistency the rate $ m = o( N^{1/2} ) $ suffices under general conditions,
see Andrews (1991) where also various choices of the weighting function are discussed.
As shown in Giraitis et al. (2003), for
Bartlett weights the rate $ m = o(N) $ suffices under certain conditions.

{\em Sequential $I(0)$ detection:} Assume the time series is $ I(1) $ before the change-point and $ I(0) $ after the change.
Noting that large values of $ U_N(s) $ provide evidence
for the unit root hypothesis whereas small values indicate $I(0)$, we propose
the stopping time
$$
  R_N = R_N(c) = \min \{ k \le n \le N : U_N( n / N ) < c \},
$$
with the convention $ \min \emptyset = N $, for some critical value (control limit) $ c $.
$k$ denotes the start of monitoring. As supported by our simulations,
one should choose $ k > 1 $ sufficiently
large to avoid that the procedure starts with only a few observations.
Thus, it is reasonable to assume that
\begin{equation}
\label{StartMonitoring}
  k = \trunc{ \kappa N }, \qquad \text{for some $ \kappa \in (0,1) $},
\end{equation}
although some of our theoretical results do not require that condition.
The related fixed-sample test rejects the null hypothesis $ H_0 $ that $ \{ Y_n \} $ is a
$ I(1) $ process in favour of the alternative $ H_a $ that the time series is $ I(0) $
if $ R_N < N $. The associated type I error rate is $ P_0( R_N < N ) $,
where $ P_0 $ indicates that the
probability is calculated assuming $ H_0 $, i.e., $ Y_n \sim I(1) $.
We propose to select $ c $ as follows. First fix size $ \alpha \in (0,1) $. Then choose $c$
such that the associated fixed-sample test has type I error rate $ \alpha $,
i.e., $ P_0( R_N < N ) = \alpha $. Our asymptotic results can be used to
obtain large sample approximations for $c$.

Noting that many classes of stationary time series considered in practice satisfy the
$ I(0) $ property (\ref{RWDefI0}), the proposed detection rule can be used to detect
stationarity, if the application suggests to consider the class of stationary $I(0)$
time series.

{\em Sequential $I(1)$ (unit root) dection:} Assume the time series starts as a (subset of a)
stationary $ I(0) $ process which ensures that the Newey-West estimator is consistent
(for conditions see Theorem~\ref{RWUnitRoot1} (ii)),
and changes its behavior to a $ I(1) $ process at a change-point (structural break).
To detect the change one may use the stopping time
$$
  \widetilde{R}_N = \widetilde{R}_N(c) = \min \{ k \le n \le N : \widetilde{U}_N( n/N ) > c \}
$$
for some critical value $ c $.
The associated fixed-sample test rejects the null hypothesis
$ H_0 : I(0) $-stationarity in favour of $ H_a : I(1) $-unit root,
if $ \widetilde{R}_N < N $.  Again, one may choose the control limit $ c $ to ensure
that the type I error $ P_0( \widetilde{R}_N < N ) $ attains a nominal value $ \alpha $.
Note that now $ P_0 $ indicates that the probability has to be calculated assuming
that $ Y_n \sim I(0) $.

\section{Asymptotic results for $ I(0) $ and $ I(1) $ processes}
\label{MainResults}

In this section we provide the asymptotic distribution theory of the processes $ U_N $
and $ \widetilde{U}_N $ and the related stopping times $ R_N $ and $ \widetilde{R}_N $
by establishing FCLTs under various basic distributional assumptions of interest. Particularly, these results can be used to obtain approximate critical values by simulating
from the limiting law and also justify to simulate the procedures using
normally distributed error terms.

\subsection{Asymptotics for $ I(1) $ processes}

The following result provides the asymptotic distribution of $ U_N $ under the random walk hypothesis that the time series $ \{ Y_n \} $ is $ I(1) $. The result in
Breitung (2002, Proposition~3, p. 349) is obtained as a special case by letting
$ K_h(\cdot) = 1 $ and $ s = 1 $. In general,
the asymptotic distribution is a functional of the Brownian motion, the kernel $ K $, and
the parameter $ \zeta = \lim_{N \to \infty} N/h_N $.

\begin{theorem}
\label{RWUnitRoot0}
Assume $ \{ Y_n \} $ is $ I(1) $ in the sense of (\ref{RWYRW}), then
\begin{equation}
\label{RWUnitRoot0a}
  U_N(s) \Rightarrow \calU_1(s) = \frac{ \zeta s^{-1} \int_0^s K(\zeta(r-s)) \left[ \int_0^r B(t) \, dt \right]^2 \, dr }
    { \int_0^s B(r)^2 \, dr },
\end{equation}
in $ D[\kappa,1] $, as $ N \to \infty $ with $ N/h \to \zeta $. The
process $ \calU_1 $ has continuous sample paths w.p. $1$.
\end{theorem}

\begin{proof} Clearly, we have
  \begin{eqnarray*}
    X_{1N}(s) &=& \trunc{Ns}^{-2} \sum_{i=1}^{ \trunc{Ns} } Y_i^2
      \Rightarrow \sigma^2 s^{-2} \int_0^s B(r)^2 \, dr = X_1(s).
   \end{eqnarray*}
   Since $ K $ is  Lipschitz continuous and $ N/h \to \zeta $, a more involved argument
   using the Skorohod/Dudley/Wichura theorem shows that
   $$
     X_{2N}(s) =
      \trunc{Ns}^{-3} \sum_{i=1}^{ \trunc{Ns} } \left( \sum_{j=1}^i Y_j \right)^2 K_h(t_i-t_{ \trunc{Ns} })
   $$
   can be approximated by a continuous functional of $ N^{-1/2} Y_{\trunc{Ns}} $, and
   therefore
   $$
   X_{2N}(s)
     \Rightarrow
     \zeta s^{-3} \sigma^2 \int_0^s \left( \int_0^r B(t) \, dt \right)^2 K( \zeta(r-s) ) \, dr = X_2(s),
   $$
  as $ N \to \infty $.
  However, to conclude weak convergence of $ X_{2N}(s) / X_{1N}(s) $ we need joint
  weak convergence of the pair $ (X_{1N}(s), X_{2N}(s)) $ in the space $ (D[\kappa,1])^2 $.
  By the Skorohod/Dudley/Wichura
  theorem we may assume that the convergence of $ X_{1N}(s) $ and $ X_{2N}(s) $ is in the supnorm.
  First, note that  the finite-dimensional distributions
  of $ \lambda_1 X_{1N} + \lambda_2 X_{2N} $ converge to the corresponding
  finite-dimensional distributions of $ \lambda_1 X_1 + \lambda_2 X_2 $, as
  $ N \to \infty $, for any scalars $ (\lambda_1,\lambda_2) \in \R^2 $. Further, clearly,
  the sequence
  $ \{ (X_{1N}(s),X_{2N}(s)) : s \in [\kappa,1] \} $, $ N \ge 1 $, is tight, since both
  coordinate processes are tight. We obtain {\em joint} weak convergence
  $ (X_{1N},X_{2N}) \Rightarrow (X_1,X_2) $, $N \to \infty $, as elements of
  the function space $ [ D([\kappa,1]) ]^2 $. Now we can apply
  the continuous mapping theorem (CMT) to obtain $ X_{2N}/X_{1N} \Rightarrow X_2 / X_1 $,
  $ N \to \infty $. Since $K$ is Lipschitz continuous and integration is continuous, numerator
  and denominator are continuous functions of $s$, w.p. $1$. Hence $
  \calU_1(s) $ has continuous sample paths w.p. $1$.
\end{proof}

\begin{remark} Observe that the theorem can be slightly extended to yield $ U_N(s) \Rightarrow \calU_1(s) $, as $ N \to \infty $, in $ D[0,1] $, if $ \, \calU_1(s) $ is defined for $ s \in (0,1] $ by the
right side of (\ref{RWUnitRoot0a}), and $ \calU_1(0) = 0 $. Nevertheless, in
Theorem~\ref{RWUnitRoot0}, and also in the sequel, for
$ U_N(s), \widetilde{U}_N(s), R_N, $ and $ \widetilde{R}_N $ we consider
weak convergence in the space $ D[\kappa,1] $, which suffices for our
purposes.
\end{remark}

In practical applications the time series is sometimes first demeaned or detrended.
This alters the asymptotic distribution as follows.

\begin{remark}
\label{Rem1} Suppose the procedure is applied to the residuals instead of the
original observations. For applications the most important cases are that the
sample $ Y_1, \dots, Y_N $ is centered at its mean or detrended. In the former case
$ Y_i $ is replaced by
  $$
    \widehat{\epsilon}_i = Y_i - N^{-1} \sum_{i=1}^N Y_i, \qquad  i = 1, \dots, N, \qquad \text{{\em ('demeaned')}},
  $$
  whereas in the latter one uses
  $$
    \widehat{\epsilon}_i = Y_i - \widehat{\beta}_0 - \widehat{\beta}_1 i, \qquad i = 1, \dots, N, \qquad
     \text{{\em ('detrended')}},
  $$
  where $ \widehat{\beta}_0 $ and $ \widehat{\beta}_1 $ are the OLS estimators from a regression of $ Y_i $
  on the regressors $ (1,i) $. Then the Brownian motion $ B $ in the formula
 for $ \calU_1 $ has to be replaced by
  the tied-down Brownian motion (Brownian bridge)
  $ B^{\mu}(s) = B(s) - s B(1) $ when demeaning
  and
  $$
    B^t(s) = B(s) - (4-6s) \int_0^1 B(r) \, dr - (12 s - 6) \int_0^1 r B(r) \, dr,
      \qquad s \in [0,1],
  $$
  when detrending.
\end{remark}

\begin{corollary} Under the conditions of Theorem~\ref{RWUnitRoot0},
$$
  N^{-1} R_N \stackrel{d}{\to} \min \{ \kappa \le s \le 1 : \calU_1(s) < c \},
  \qquad N \to \infty.
$$
\end{corollary}

\begin{proof} Note that
$$
  N^{-1} R_N > x \Leftrightarrow \sup_{s \in [\kappa,x]} U_N(s) \ge c
$$
By the CMT
$$
  V_N(x) = \sup_{s \in [\kappa, x]} U_N(s)
  \Rightarrow
  \sup_{s \in [\kappa,x]} \calU_1(s) = V(x),
$$
$ N \to \infty $, which implies that
for all $ x \in \R $ and all continuity points $ c > 0 $ of the
distribution function of $ V(x) $ we have
$$
  \lim_{N\to\infty} F_N(x) = \lim_{N\to\infty} P( N^{-1} R_N \le x )
  =  P( \sup_{s \in [\kappa,x]} \calU_1(s) < c ) = F(x).
$$
It remains to check whether $ V(x) $ may have atoms. Since $ \calU_1 \in C[0,1] $
w.p. $1$, we may work in the separable Banach space $ ( C[0,1], \| \circ \|_\infty )$
and can apply Lifshits (1982, Theorem 2) which asserts that the distribution of $ V(x) $
can have an atom only at the point
$$
  \gamma_x = \sup_{0 \le t \le x:\ \operatorname{Var} \calU_1(t) = 0}E( \calU_1(t) ) = 0,
$$
equals $0$ on $ (-\infty, \gamma_x) $, and is absolutely continuous with respect to
Lebesgue measure on $ [\gamma_x, \infty) $. Hence all $ c > 0 $ are continuity points
of $ V(x) $.
\end{proof}

\subsection{Asymptotics for $ I(0) $ processes}

For weakly stationary $I(0)$ time series $ \{ Y_n \} $ satisfying a certain condition
on the fourth-order moments, the variance-ratio process $ U_N(s) $
still has a non-degenerate limiting distribution if scaled with $ N $. However,
the limit depends on a nuisance parameter summarizing the dependence structure.
We consider both mean-zero $ I(0) $
processes and $ I(0) $ processes which are disturbed by a local deterministic alternative.

For the process $ \widetilde{U}_N(s) $ using a Newey-West type estimator to eliminate
the nuisance parameter from the limit distribution, we use a weak mixing condition.

The first result considers stationary $I(0)$ processes. The limit
given in Kwiatkowski et al. (1992, formula 14) is obtained as a special case if
$ K_h(\cdot) = 1 $, $ \zeta = 1 $, and $ s = 1 $.

\begin{theorem}
\label{RWUnitRoot1}
\begin{itemize}
\item[(i)]
Assume $ \{ Y_n \} $ is a weakly stationary mean zero $ I(0) $ process such that
$ \{ Y_n^2 \} $ is weakly stationary,
\begin{equation}
\label{Cond4}
  \gamma_2(k) = E( Y_1^2 Y_{1+k}^2 ) \to 0, \qquad |k| \to \infty.
\end{equation}
Then
$$
  \trunc{Ns} U_N(s)
    \Rightarrow
   \calU_2(s) = \frac{\sigma^2}{EY_1^2} s^{-1} \zeta \int_0^s B(r)^2 K( \zeta(r-s) ) \, dr,
$$
in $ D[\kappa,1] $, as $ N \to \infty $. The process $ \calU_2 $ has continuous sample paths w.p. $1$.
\item[(ii)] Assume $ \{ Y_n \} $ is a strictly
stationary $ \alpha $-mixing $ I(0) $ process such that $ E Y_1^{4 \nu} < \infty $ and
\begin{equation}
\label{Mixing2}
  \sum_{j=1}^\infty j^2 \alpha(j)^{(\nu-1)/\nu} < \infty,
\end{equation}
for some $ \nu > 1 $. Then, if $ m/N^{1/2} = o(1) $,
$$
  \widetilde{U}_N(s) \Rightarrow \widetilde{\calU}_2(s) = s^{-1} \zeta \int_0^s B(r)^2 K( \zeta(r-s) ) \, dr,
$$
in $ D[\kappa,1] $, as $ N \to \infty $.
\end{itemize}
\end{theorem}

\begin{proof} By assumption
$
  N^{-1/2} \sum_{i=1}^{\trunc{Ns}} Y_i \Rightarrow \sigma B(s),
$
as $ N \to \infty $, where $ \sigma^2 = \sum_{k=-\infty}^\infty E(Y_1 Y_{1+k} ) $.
By the weak law of large numbers (Brockwell and Davis (1991), Theorem 7.1.1),
(\ref{Cond4}) implies that for fixed $ s \in [\kappa,1] $
$$
  V_{N1}(s) = \trunc{Ns}^{-1} \sum_{i=1}^{\trunc{Ns}} Y_i^2 \stackrel{P,L_2}{\to} E Y_1^2 = V_1,
$$
as $ N \to \infty $, where $ \stackrel{P,L_2}{\to} $ means that the convergence holds
in probability and in quadratic mean, i.e., in the $L_2$-space. The limit is
a.s. constant in $s$. Further,
\begin{eqnarray*}
  V_{N2}(s) &=& \frac{N}{\trunc{Ns}} N^{-1} \sum_{i=1}^{\trunc{Ns}}
   \left( \sum_{j=1}^i Y_j \right)^2 K_h( t_i - t_{\trunc{Ns}})  \\
  &=& \frac{N}{\trunc{Ns}}
        (N/h) \int_0^s \left( N^{-1/2} \sum_{j=1}^{\trunc{Nr}} Y_j \right)^2
          K( \trunc{Nr}/h - \trunc{Ns}/h ) \, dr \\
  & \Rightarrow &
    \sigma^2 s^{-1} \zeta  \int_0^s K(\zeta(r-s)) B(r)^2 \, dr = V_2(s),
\end{eqnarray*}
if $ N \to \infty $. Billingsley (1968, Theorem 4.4) now implies
weak convergence of the pair $ (V_{N1}, V_{N2}) $, and an application
of the CMT yields $ V_{N2} / V_{N1} \Rightarrow V_2 / V_1 = \calU_2 $, as
$ N \to \infty $.
Since $ K $ is Lipschitz continuous, the process $ \calU_2 $ has continuous sample paths w.p. $1$. To show (ii) the proof is modified as follows. By Andrews (1991, Lemma~1)
the mixing condition (\ref{Mixing2}) ensures his Assumption A.
Hence, if $ m/N^{1/2} = o(1) $,
$ s^2_{Nm}(s) \stackrel{P}{\to} s \sigma^2 $, as $ N \to \infty $, which implies
weak convergence to the non-stochastic function $ s \sigma^2 $, $ s \in [\kappa,1] $.
\end{proof}

\begin{remark} Statement (i) implies that the $ U_N $
statistic is consistent against stationary alternatives.
\end{remark}

\begin{remark}
Consistency of $ s_{Nm}^2 $ for Bartlett weights
has also been shown under the weaker condition $ m/N = o(1) $ provided that
$ \sum_j | \gamma_j | < \infty $, $ \gamma_j = \Cov(Y_1,Y_{1+j}), $ and
$$
  \sup_h \sum_{r,s = -\infty}^{\infty} | \kappa(h,r,s) | < \infty
$$
where
$$
  \kappa(h,r,s) = E[(Y_k-\mu)(Y_{k+h}-\mu)(Y_{k+r}-\mu)(Y_{k+s}-\mu)]
   - ( \gamma_h \gamma_{r-s} + \gamma_r \gamma_{h-s} + \gamma_s \gamma_{h-s} )
$$
is the fourth order cumulant (Giraitis et al. (2003), Theorem~3.1.). This condition holds, e.g.,
for linear processes with absolutely summable coefficients.
\end{remark}

\begin{remark} In case that the time series is demeaned or detrended first, again the Brownian motion
  in the representation of $ \calU_2 $ has to be replaced by the tied-down Brownian motion $ B^{\mu} $
  or the process $ B^t $.
\end{remark}


Again, we have the following corollary for the related stopping time.

\begin{corollary} Under the conditions of Theorem~\ref{RWUnitRoot1} we have
$$
  N^{-1} \widetilde{R}_N \stackrel{d}{\to}
    \min \{ \kappa \le s \le 1 : \widetilde{\calU}_2(s) > c \},  \qquad N \to \infty .
$$
\end{corollary}

So far we considered mean-zero time series.
The following theorem provides sufficient conditions for a well-defined limit
for a $I(0)$ series with a (local) nonparametric trend.

\begin{theorem}
\label{RWUnitRoot2}
Suppose $ \{ Y_n \} $ satisfies $ Y_n = m_n + u_n $,$ n \in \N $, $h > 0 $, where
$ \{ u_n \} $ satisfies the conditions of Theorem~\ref{RWUnitRoot1}~(i) with the $ Y_n $
replaced by $ u_n$'s,
and $ \{ m_{N,n} \} $ is an array of non-negative constants with
$ m_{N,n} \le M $ for all $N,n$, such that
\begin{itemize}
  \item[(i)] $ N^{-1/2} \sum_{i=1}^{ \trunc{Ns} } u_i \Rightarrow \sigma B(s) $, as $ N \to \infty $,
    for some $ 0 < \sigma < \infty $,
  \item[(ii)] $ \sup_{0 \le s \le 1} | N^{-1/2} \sum_{i=1}^{ \trunc{Ns} } m_{N,i} - \mu(s) | \to 0 $, as $N \to \infty$.
\end{itemize}
for some deterministic drift $ \mu(s) \in D[0,1] $, which is continuous at $0$. Then
$$
  \trunc{Ns} U_N(s) \Rightarrow \calU_2^\mu(s)
    = \frac{1}{s E(Y_1^2) } \zeta \int_0^s [\mu(r)+\sigma B(r)]^2 K( \zeta(r-s) ) \, dr,
    \qquad \text{in $D[\kappa,1]$},
$$
as $ N \to \infty $. If $ \mu \in C[0,1] $, then $ \calU_2^\mu $ has continuous sample paths w.p. $1$.
\end{theorem}

\begin{proof} Using Jacod and Shiryaev (2003, VI, Proposition~1.22, p.~329) conditions (i) and (ii) yield
$
  N^{-1/2} \sum_{i=1}^{ \trunc{Ns} } Y_i \Rightarrow \mu(s) + \sigma B(s),
$
in $ D[0,1] $, as $ N \to \infty $.
The proof follows by a simple modification of the proof of the previous theorem by
noting that for the denominator we have
\begin{eqnarray*}
  \trunc{Ns}^{-1} \sum_{i=1}^{ \trunc{Ns} } Y_i^2
   & = &
        \trunc{Ns}^{-1} \sum_{i=1}^{\trunc{Ns}} u_i^2
        + 2 \trunc{Ns}^{-1} \sum_{i=1}^{\trunc{Ns}} m_{N,i} u_i
        + \trunc{Ns}^{-1} \sum_{i=1}^{\trunc{Ns}} m_{N,i}^2 \\
   & \stackrel{P}{\to} & EY_1^2,
\end{eqnarray*}
as $ N \to \infty $, because $ m_{N,i} \le M $ for all $i$ and $N$
for some constant $M >0$
implies $ N^{-1} \sum_i m_{N,i}^2 \le M/N^{1/2} \sum_i (m_{N,i}/N^{1/2}) = o(1) $.
\end{proof}

We illustrate the conditions (i) and (ii) by a local change-point model,
where starting at a change-point $ \trunc{N \vartheta} $ the mean is no longer $0$
but positive and induced by a non-negative function $ m_0 : \R \to \R^+ $.
Particularly, (truncated) linear trends as
$ m_0(x) = a x $ if $ x \in [0,1] $ and $ m_0(x) = 0 $ otherwise for
some $ a > 0 $ are allowed.

\begin{example} Assume $ Y_{N,n} = m_{N,n} + u_n $ with
$$
  m_{N,n} = m_0( (n - \trunc{N\vartheta})/N ) N^{-1/2}.
$$
Here we assume that the function $ m_0 : \R \to \R $ satisfies $ m(s) = 0 $ for $ s<0$, is right-continuous, non-negative with bounded variation, and has at most finitely many jumps. Further, $ m_0 $ is assumed to be Lipschitz continuous and bounded between the jumps, and is integrable, i.e., $ \int_0^\infty m_0(t) \, dt < \infty $. Finally, we assume that
there is some $ t^* > 0 $ with $ m_0(t) > 0 $ for $ t \in (0,t^*) $.
It is easy to see that the conditions (i) and (ii) of Theorem~\ref{RWUnitRoot2} are
satisfied. The limiting mean function is given by
$$
  \mu(s) = \int_0^s m_0( r - \vartheta ) \, dr,
$$
and one obtains
$$
  N^{-1/2} \sum_{i=1}^{ \trunc{Ns} } Y_{N,i} \Rightarrow \int_0^s m_0(r-\vartheta) \, dr
   + \sigma B(s),
$$
as $ N \to \infty $.
\end{example}

\begin{remark}
By making use of the Karhunen-Lo\`eve representation
$$
  B(t) \stackrel{d}{=}  \sqrt{2} \sum_{n=0}^\infty \frac{ \sin( (n-1/2) \pi t ) }
                    { (n-1/2) \pi  } Z_n, \ t \in [0,1],
$$
where $ \{ Z_n \} $ are i.i.d. $ N(0,1) $-random variables, cf.
Ito and Nisio (1968), we also represented the limiting distributions
as simple rational functions of infinite quadratic forms of the type
$$ \sum_{n,m=0}^\infty \gamma_{mn}(s) Z_n Z_m. $$
Following a referee, we omit these results here,
since we did not use them for our simulations.
\end{remark}

\subsection{Asymptotics for local-to-unity processes}

Let us now consider the asymptotic behaviour of $ U_N $ under a model which
is nearly $ I(1) $. More precisely, we consider a sequence of models which converges to an
$ I(1) $ model yielding what is called {\em local-to-unity asymptotics}.
The {\em local-to-unity model} assumes that we are given an array
$ \{ Y_{N,n} : N \in \N, n \in \N \} $ satisfying
\begin{equation}
\label{RWLocalToUnityModel}
  Y_{N,n+1} = (1 + a/N) Y_{N,n} + u_n, \qquad 1 \le n \le N, \ N \in \N,
\end{equation}
where $ a \in \R $ and $ \{ u_n \} $ is a $ I(0) $ process.
Thus, $ \{ Y_{N,n} \} $ converges to a random walk,  as $ N \to \infty $.
Note that both positive and negative values for $ a $ are allowed.
Under the local-to-unity model an Ornstein-Uhlenbeck process
appears in the limit process instead of the Brownian motion.
It has been proposed in the literature to use estimates for $a$ and to use
the corresponding asymptotic distributions under the local-to-unity asymptotics
as approximations. The analyses of Stock and Watson (1998, Table~7)
for the US annual series of the GDP, consumption, and investment imply estimates for
$a$ in the region between $-15$ and $-3$.

The following theorem contains Breitung (2002, Proposition~4, p. 350) as a special case.

\begin{theorem}
\label{RWUnitRootLTU}
Assume the local-to-unity model (\ref{RWLocalToUnityModel}) holds. Then
$$
  U_N(s) \Rightarrow \calU_\calZ(s) = \frac{ \zeta s^{-1} \int_0^s \left[ \int_0^r \calZ(t;a) \, dt
  \right]^2 K(\zeta(s-r)) \, dr, }
                          { \int_0^s \calZ(r;a)^2 \, dr },
$$
in $ D[\kappa,1] $, as $ N \to \infty $,
where
$$
  \calZ(s;a) = \int_0^s e^{a(s-r)} \, d B(r),\qquad s \in [0,1],
$$
is an Ornstein-Uhlenbeck process. Further,
$$
  N^{-1} R_N \stackrel{d}{\to} \min \{ \kappa \le s \le 1 : \calU_\calZ(s) < c \},
  \qquad N \to \infty.
$$
\end{theorem}

\begin{remark} The stochastic integral appearing here is
of the type $ I(s) = \int_0^s F(s,r) \, B(dr) $, $ F(s,r) $ non-stochastic, strictly
monotone in $r$ and bounded  with bounded variation.
It is a special case of the Ito integral.
However, since the Stieltjes integral
$ \int_0^s B(r) \, F(s,dr) $ exists, $I(s) $ can also be defined
by the integration by parts formula
$$
  \int_0^s B(r) \, F(s,dr) = F(s,r) B(r) |_{r=0}^{r=s}
  - \int_0^s F(s,r-) \, B(dr),
$$
where $ F(s,r-) = \lim_{x \uparrow r} F(s,r) $. For this approach see
Shorack and Wellner (1986, p.~127) or Gill (1989, p.~110).
\end{remark}

\begin{proof}
Put $ W_n = \sum_{i=1}^n u_i $, $n \in \N $,
and $ S_N(s) = N^{-1/2} \sum_{i=1}^{ \trunc{Ns} } u_i $, $ s \in [0,1].$
By assumption $ S_N(s) \Rightarrow \sigma B(s) $ in $ D[0,1] $, as $ N \to \infty $.
We may assume $ \sigma = 1 $.
Note that $ N^{-1/2} Y_{ N, \trunc{Ns} } $ can be written as a stochastic
Stieltjes integral, namely
\begin{eqnarray*}
  N^{-1/2} Y_{ N, \trunc{Ns} } & = &
    N^{-1/2} \sum_{i=1}^{ \trunc{Ns} } (1 + a/N)^{ \trunc{Ns} - i } u_i \\
  & = &
     \sum_{i=1}^{ \trunc{Ns} } (1 + a/N)^{ \trunc{Ns} - i } N^{-1/2}(W_i - W_{i-1})
    = \int_0^s e_N(r;s) \, d S_N(r),
\end{eqnarray*}
where the integrand,
$ e_N(r;s) = (1+a/N)^{\trunc{Ns} - \trunc{Nr} }, (r,s) \in \calI = \{ (u,v) \in [0,1]^2 :
0 \le u \le v, 0 \le v \le 1 \} $, is a step function in $ r $. The fact that
$ \log[ (1+a/N)^{\trunc{Ns}-\trunc{Nr}}/e^{a(s-r)} ] = a(s-r) + o(1) $
uniformly in $ (r,s) \in \calI $ implies that $ e_N(r;s) $ converges uniformly
in $ (r,s) \in \calI $ to $ e(r;s) = e^{a(s-r)} $. Particularly, there is some
constant $ C$ such that $ | e_N(r;s) | \le C $ for all $ (r,s) \in \calI $.
Further, since for $ a < 0 $ and fixed $ s \in [0,1] $,
the variation $ \int | d e_N(\cdot;s) | $ of
$ e_N(r;s) $ as a function of $ r \in [0,s] $ is given by
\begin{eqnarray*}
  \int | d e_N(\cdot;s)| &=& \sum_{i=0}^{\trunc{Ns}-1} \left\{ e_N[(i+1)/N;s] - e_N( i/N; s ) \right\} \\
  &=& (1+a/N)^{\trunc{Ns} - \trunc{Nr}} - (1+a/N)^{\trunc{Ns}}
   \to e^{a(s-r)} - e^{as}.
\end{eqnarray*}
It follows that $ \sup_{s \in [0,1]} \int | d e_N(\cdot;s) | < \infty $.
To estimate
$$ \sup_{s \in [0,1]} \left| \int_0^s e_N(s;r) \, dS_N(r)
 - \int_0^s e^{a(s-r)} \, d B(r) \right|,
$$ we use the decomposition
$$
  \int_0^s e_N(r;s) \, dS_N(r) - \int_0^s e(r;s) \, d B(r)
    = \int_0^s [e_N(r;s) - e(r;s)] \, dB(r) + \int_0^s e_N(r;s) \, d[S_n(r)-B(r)].
$$
Of course,
$$
  \sup_{s \in [0,1]} \left| \int_0^s e_N(r;s) \, dB(r) - \int_0^s e(r;s) \, d B(r) \right|
  \stackrel{P}{\to} 0,
$$
as $ N \to \infty $, see, e.g., Shorack and Wellner (1986, p.~130).
Integration by parts yields
$$
  \sup_{s \in  [0,1]} \left| \int_0^s e_N(r;s) \, d [S_N(r) - B(r)] \right|
 \le 2 C \| S_N - B \|_\infty  +
 \| S_N - B \|_\infty \sup_{s \in [0,1]} \int | d e_N(\cdot;s) |.
$$
These estimates imply
$$
  N^{-1/2} Y_{ N, \trunc{Ns} }
  = \int_0^s e_N(r;s) \, d S_N(r) \Rightarrow \int_0^s e^{a(s-r)} \, d B(r)
  =  \calZ(s;a),
$$
as $ N \to \infty $. We obtain
\begin{eqnarray*}
  V_{N1}(s) &=& \frac{N^2}{\trunc{Ns}^2}
  N^{-2} \sum_{i=1}^{ \trunc{Ns} } Y_{N,i}^2  \\
  &=& \frac{N^2}{\trunc{Ns}^2} \int_0^s (N^{-1/2} Y_{ N, \trunc{Nr} } )^2 \, dr \\
  &\Rightarrow& \sigma^2 s^{-2} \int_0^s \calZ(r;a)^2 \, dr = V_1(s),
\end{eqnarray*}
in $ D[\kappa,1] $, as $ N \to \infty $, and, using the same arguments as in the proof of Theorem~\ref{RWUnitRoot0}
\begin{eqnarray*}
 V_{N2}(s)  &= & \frac{N^3}{\trunc{Ns}^3}
                 N^{-3} \sum_{i=1}^{ \trunc{Ns} } \left( \sum_{j=1}^i Y_{N,j} \right)^2 K_h(t_i-t_{ \trunc{Ns} }) \\
      && \qquad = \frac{N^3}{\trunc{Ns}^3}
                  N^{-2} \int_0^s \left( \sum_{j=1}^{ \trunc{Nr} } Y_{N,j} \right)^2
                            K_h( t_{ \trunc{Nr} } - t_{ \trunc{Ns} } ) \, dr \\
      && \qquad = \frac{N^3}{\trunc{Ns}^3}
             N h^{-1} \int_0^s \left( \int_0^r N^{-1/2} Y_{ N, \trunc{ Nt } } \, dt \right)^2
              K( \trunc{Nr}/h - \trunc{Ns}/h ) \, dr  \\
      && \qquad \Rightarrow
            \zeta \sigma^2 s^{-3} \int_0^s \left( \int_0^r \calZ(t;a) \, dt \right)^2 K(\zeta(s-r)) \, dr = V_2(s),
\end{eqnarray*}
in $ D[\kappa,1] $,  as $ N \to \infty $.
Noting that $ V_{N1} $ and $ V_{N2} $ are functionals
of $ N^{-1/2} Y_{N,\trunc{Ns}} $ up to negligible terms, we obtain
$ (V_{N1}, V_{N2}) \Rightarrow (V_1,V_2) $, as $ N \to \infty $. Hence the
CMT yields the assertion.
\end{proof}

\section{Change-point models}
\label{SecCP}

The question arises how the sequential processes and stopping times considered above behave
under a change-point model where after a certain fraction of the data the
time series changes. In this section we consider both change-point models, a change
from $ I(0) $ to $ I(1) $ and a change from $ I(1) $ to $ I(0) $. To design a
monitoring procedure (stopping rule) having well-defined properties under the
null hypothesis of no change, the results from the previous section about the
asymptotic distribution of $ U_N(s) $ (under a $ I(1) $ process) and $ \trunc{Ns} U_N(s) $
or $ \widetilde{U}_N(s) $ (under a $ I(0) $ process) apply.
In particular, for the $ I(0) $-to-$I(1)$ change-point
model monitoring can be based on the stopping time $ \widetilde{R}_N $ calculated from
the process $ \widetilde{U}_N(s) $ which has the well-defined limit
$ \widetilde{\calU}_2(s) $ for $ I(0) $ processes, i.e., under the null hypothesis of no change.
To design a stopping rule for a $ I(1) $-to-$ I(0) $ change-point model one would
rely on the stopping time $ R_N $ and its asymptotic distribution, which is a functional
of $ \calU_1(s) $, the well-defined limit of $ U_N(s) $ for $I(1)$ processes.

We will now study the asymptotic laws under the general case of a change, i.e.,
under the alternative hypothesis that a change-point exists.
For both change-point
models the integrated subseries of the time series $ Y_1, \dots, Y_N $ determines
the proper scaling, since in both models $ U_N(s) $ has a well-defined limit,
whereas $ \trunc{Ns} U_N(s) $ is degenerated. In this sense,
the change-point problems are qualitatively different and the situation is not symmetric.
The $ I(1) $-to-$I(0)$ is smoother in the sense that the same process can be
considered to study the behaviour under the no-change hypothesis and the alternative
of a change-point, whereas for the $ I(0)$-to-$I(1)$ problem the scaling has to be changed.

\subsection{A change from $I(0)$ to $ I(1) $}

Let us assume that the time series is a mean-zero fourth-order $ I(0) $ process
satisfying condition (\ref{Cond4}) at the beginning,
but becomes a random walk, i.e., $ I(1) $ process, starting at the time point
$ \trunc{N \vartheta} $, where $ \vartheta \in (0,1) $.
We consider the following change-point model. Let $ \{ u_n : n \ge 0 \} $ be a mean-zero weakly  stationary $ I(0) $ time series such that all moments of order $4$ exists and
are stationary,  and condition (\ref{Cond4}) is satisfied when the $ Y_n $ are replaced by $ u_n $'s. Further assume that
\begin{equation}
\label{I0toI1}
  Y_i = \left\{ \begin{array}{ll} u_i, & i = 1, \dots, \trunc{N \vartheta}-1 \\
                  Y_{i-1} + u_i, & i = \trunc{N\vartheta}, \dots, N.
                \end{array}
        \right.
\end{equation}
Then, $ Y_1, \dots, Y_{ \trunc{N \vartheta}-1} \sim I(0) $ and
$ Y_{ \trunc{N \vartheta} }, \dots, Y_N \sim I(1) $, i.e., at the change-point
$ \trunc{N\vartheta} $ the time series changes from stationarity to a $ I(1) $ series.

\begin{lemma}
\label{I0toI1Lemma} Under the change-point model (\ref{I0toI1}) we have
$$
  \int_0^s N^{-1/2} Y_{ \trunc{Nt} } \, dt
    \Rightarrow \sigma \int_{\min(s,\vartheta)}^s B(t) \, dt,
  \quad N \to \infty.
$$
in $ D[0,1] $.
\end{lemma}

\begin{proof}
Define
$$
  W_N(s) = \left\{
    \begin{array}{ll}
      \int_0^s \sqrt{N} Y_{ \trunc{Nt} } \, dt, & 0 \le s < \vartheta, \\
        \int_0^\vartheta \sqrt{N} Y_{\trunc{Nt}} \, dt
        + N^{-1/2} ( Y_{\trunc{Ns}} - Y_{\trunc{N\vartheta}} ),
          & \vartheta \le s \le 1.
    \end{array} \right.
$$
Then, $ W_N \in D[0,1] $ for each $ N \in \N $. By construction of $ W_N $, we have
$ W_N \Rightarrow \sigma B $, $N \to \infty $, in $ D[0,1] $, if (\ref{I0toI1}) holds.
By the Dudley/Skorohod/Wichura representation theorem in general metric spaces,
there exists a probability space with equivalent versions
$ \widetilde{W}_N $ and $ \widetilde{B} $ such that
$ d( \widetilde{W}_N, \widetilde{B} ) \to 0 $ a.s, for $ N \to \infty $. Since
$ \widetilde{B} \in C[0,1] $, we even have
$ \| \widetilde{W}_N - \widetilde{B} \|_\infty \to 0 $ a.s., $ N \to \infty $. Thus, we may assume
$$
  \sup_{0 \le s \le \vartheta} \biggl| N^{1/2} \int_0^s Y_{\trunc{Nt}} \, dt - \sigma B(s) \biggr|
    \stackrel{a.s.}{\to} 0, \ N \to \infty ,
$$
and
$$
  \sup_{\vartheta \le s < 1} | N^{-1/2} ( Y_{\trunc{Ns}} - Y_{\trunc{N\vartheta}} )
    - \sigma [B(s) - B(\vartheta)] | \stackrel{a.s.}{\to} 0, \ N \to \infty.
$$
Note that for $ \vartheta \le s \le 1 $ we have
\begin{eqnarray*}
  \int_0^s N^{-1/2} Y_{\trunc{Nt}} \, dt &=&
    N^{-1} \int_0^\vartheta N^{1/2} Y_{\trunc{Nt}} \, dt
    + \int_\vartheta^s N^{-1/2}( Y_{\trunc{Nt}} - Y_{\trunc{N\vartheta}} ) \, dt \\
   &&\quad + N^{-1/2} Y_{\trunc{N\vartheta}} (s-\vartheta),
\end{eqnarray*}
which should be close to
$
  \sigma B(s)/N + \sigma \int_\vartheta^s B(t) \, dt,
$
whereas for $ 0  \le s < \vartheta $ the second and third term vanish.
Indeed, if we define the $ D[0,1] $-valued process
$$
  A_N(s) = \left\{
    \begin{array}{ll}
      \sigma B(s) / N, & 0 \le s < \vartheta, \\
      \sigma B(\vartheta)/N + \sigma \int_\vartheta^s B(t) \, dt,
      & \vartheta \le s \le 1,
    \end{array} \right.
$$
and observe that $ A_N(s) \Rightarrow A(s;\vartheta) $, as $ N \to \infty $, where
$$
  A(s;\vartheta) = \sigma \int_{\min(s,\vartheta)}^s B(t) \, dt, \qquad 0 \le s \le 1,
$$
we obtain the estimate
\begin{eqnarray*}
  \sup_{s \in [0,1]} \biggl| \int_0^s N^{-1/2} Y_{\trunc{Nt}} \, dt - A_N(s) \biggr|
  &\le& N^{-1} \sup_{s\in [0,\vartheta)} \biggl| \int_0^s N^{1/2} Y_{\trunc{Nt}} \, dt
           - \sigma B(s) \biggr| \\
  && + N^{-1} \sup_{s \in [\vartheta,1]} \biggl| \int_0^\vartheta N^{1/2} Y_{\trunc{Nt}} \, dt - \sigma B(\vartheta) \biggr| \\
  && + \sup_{s \in [\vartheta,1]} \biggl| \int_\vartheta^s N^{-1/2}
    \{ Y_{\trunc{Nt}} - Y_{\trunc{N\vartheta}} \} \, dt - \int_\vartheta^s \sigma B(t) \,dt
     \biggr| \\
  && + \sup_{s \in [\vartheta,1]} \biggl| N^{-1/2} Y_{\trunc{N\vartheta}} (s-\vartheta)
    \biggr|
  \stackrel{a.s.}{\to} 0,
\end{eqnarray*}
as $ N \to \infty $. Whereas the first three terms are obvious, let us consider the last one.
According to (\ref{I0toI1}) we have
$$
  m_4 = E | Y_{\trunc{N\vartheta}} |^4 = E | u_{\trunc{N\vartheta}-1} + u_{\trunc{N\vartheta}} |^4
  < \infty,
$$
since $ E | u_1 |^4 < \infty $ by assumption. Therefore,
\begin{eqnarray*}
  E \left( \sup_{s \in [\vartheta,1]} \left| N^{-1/2} Y_{\trunc{N\vartheta}} (s-\vartheta) \right| \right)^4
  & = & E [ N^{-1/2} | Y_{\trunc{N\vartheta}} | (1-\vartheta) ]^4 \\
  & \le & \frac{m_4 (1-\vartheta)^4 }{ N^2 }.
\end{eqnarray*}
By Markov's inequality we can conclude that for any $ \varepsilon > 0 $
$$
  \sum_{N=1}^\infty P\left( \sup_{s \in [\vartheta,1]} \left| N^{-1/2} Y_{\trunc{N\vartheta}} (s-\vartheta) \right| > \varepsilon \right)
  \le \sum_{N=1}^\infty \frac{m_4 (1-\vartheta)^4}{\varepsilon^4 N^2} < \infty.
$$
Now Serfling (1980, Theorem~1.3.4) yields
$$
  \sup_{s \in [\vartheta,1]} | N^{-1/2} Y_{\trunc{N\vartheta}} (s-\vartheta) |
  \stackrel{a.s.}{\to} 0,
$$
as $ N \to \infty $. The convergence
$$
  \sup_{s \in [0,1]} \biggl| \int_0^s N^{-1/2} Y_{\trunc{Nt}} \, dt - A_N(s) \biggr|
  \stackrel{a.s.}{\to} 0,
$$
as $ N \to \infty $, implies convergence in the metric $d$, which
in turn implies weak convergence of the original versions,
see e.g. Billingsley (1968, Theorem 4.3) or van der Vaart (1998, Theorem 18.10).
\end{proof}


\begin{theorem}
\label{CPI0toI1}
Assume the $ I(0) $-to-$I(1)$ change-point model
 (\ref{I0toI1}) holds. Then for $ U_N $ as defined in
 (\ref{DefUN}) we have
 $$
   U_N(s) \Rightarrow \calU_{01,\vartheta}(s) =
     \left\{ \begin{array}{ll} 0,   &  s \in [0,\vartheta), \\
     \frac{ s^{-1} \zeta \int_0^1 \eins(r \ge \vartheta) [
              \int_\vartheta^r B(t) \, dt ]^2
             K( \zeta(r-s) ) \, dr }
          { \eins(s \ge \vartheta) \int_\vartheta^s [ B(t) + B(\vartheta)]^2 \, dt },
      & s \in [\vartheta,1],
     \end{array} \right.
 $$
 as $ N \to \infty $, yielding
 $$
   N^{-1} R_N \stackrel{d}{\to} \min \{ \kappa \le s \le 1 : \calU_{01,\vartheta}(s) > c \},
 $$
 as $ N \to \infty $.
\end{theorem}

\begin{proof} An easy application of Lemma~\ref{I0toI1Lemma}
  yields for the numerator of $ U_N $
  \begin{eqnarray*}
    && \trunc{Ns}^{-3} \sum_{i=1}^{ \trunc{Ns} } \biggl( \sum_{j=1}^i Y_j \biggr)^2
       K_h( t_i - t_{\trunc{Ns}} ) \\
    && \qquad = \frac{ N^3 }{ \trunc{Ns}^3 } \frac{N}{h}
     \int_0^s \biggl( \int_0^r N^{-1/2} Y_{\trunc{Nt}} \, dt \biggr)^2
        K [(\trunc{Nr}-\trunc{Ns})/h] \\
    && \qquad \Rightarrow
      s^{-3} \zeta \int_\vartheta^s
        \biggl[ \sigma \int_\vartheta^r B(t) \, dt \biggr]^2
        K( \zeta(r-s) ) \, dr,
  \end{eqnarray*}
  in $ D[\kappa,1] $, as $ N \to \infty $.
  Since for $ s \ge \vartheta $
  $$
    \trunc{Ns}^{-2} \sum_{i=\trunc{N\vartheta}}^{\trunc{Ns}} Y_i^2
    = \frac{N^2}{\trunc{Ns}^2} N^{-1} \int_\vartheta^s Y_{\trunc{Nt}}^2 \, dt
    \Rightarrow s^{-2} \sigma^2 \int_\vartheta^s [B(t) + B(\vartheta)]^2 \, dt,
  $$
  as $ N \to \infty $, and, by stationarity of $ Y_1, \dots, Y_{\trunc{N\vartheta}-1} $,
  $ N^{-1} \sum_{i=1}^{\trunc{Ns}} Y_i^2 \to s \sigma^2 $ if $ s < \vartheta $,
  for the denominator of $ U_N $ we obtain
  $$
    \trunc{Ns}^{-2} \sum_{i=1}^{\trunc{Ns}} Y_i^2
    \Rightarrow \eins(s \ge \vartheta) s^{-2} \sigma^2
      \int_\vartheta^s [B(t) + B(\vartheta)]^2 \, dt,
  $$
  in $ D[\kappa,1] $, as $ N \to \infty $. Note that the denominator is positive w.p. $1$.
  Using the arguments given in detail in the proof of
  Theorem~\ref{RWUnitRoot0} and applying the CMT yields the assertions.
\end{proof}

\subsection{A change from $I(1)$ to $ I(0) $}

Now assume that the first part of the time series is a random walk, i.e. $ I(1) $, and
changes to a $ I(0) $ process at the change-point
$ \trunc{N\vartheta} $ for some fixed constant $ \vartheta \in (0,1) $, i.e.,
\begin{equation}
\label{I1toI0}
  Y_i = \left\{
  \begin{array}{ll} \sum_{j=0}^i u_j, & i = 0, \dots, \trunc{N\vartheta}-1, \\
                    \eta u_i, & i = \trunc{N\vartheta}, \dots, N,
  \end{array} \right.
\end{equation}
Here $ \eta > 0 $ is a scale parameter, which is briefly discussed at the end of this section,
and $ \{ u_n \} $ is a weakly stationary mean zero $I(0)$ time series satisfying
condition (\ref{Cond4}).

Model (\ref{I1toI0}) implies that the
variance function is linear with positive slope before the change and
constant after the change.

\begin{lemma}
\label{I1toI0Lemma} Assume the change-point model (\ref{I1toI0}) holds. Then
we have
$$
  \int_0^s N^{-1/2} Y_{ \trunc{Nt} } \, dt
    \Rightarrow  \sigma \int_0^{\min(s,\vartheta)} B(t) \, dt,
  \quad N \to \infty.
$$
\end{lemma}

\begin{proof} Noting that
$ N^{-1/2} Y_{\trunc{Nt}} \Rightarrow \sigma B(t) $ if $ t < \vartheta $,
$ \int_\vartheta^s N^{1/2} Y_{\trunc{Nt}} \, dt \Rightarrow \eta \sigma [B(s)-B(\vartheta)] $
if $ \vartheta \le s \le 1 $, and
$$
  \int_0^s N^{-1/2} Y_{\trunc{Nt}} \, dt
  = \int_0^\vartheta N^{-1/2} Y_{\trunc{Nt}} \, dt
  + N^{-1} \int_\vartheta^s N^{1/2} Y_{\trunc{Nt}} \, dt,
$$
if $ \vartheta \le s \le 1 $,
the lemma is shown analogously to Lemma~\ref{I0toI1Lemma}, if we define
the $ D[0,1] $-valued process
$$
  \widetilde{Z}_N(s) = \left\{
   \begin{array}{ll}
     \sigma \int_0^s B(t) \, dt,           & 0 \le s < \vartheta, \\
     \sigma \int_0^\vartheta B(t) \, dt
       + \sigma \eta /N [ B(s) - B(\vartheta) ], & \vartheta \le s \le 1,
   \end{array}
   \right.
$$
and note that
$
  \widetilde{Z}_N(s) \Rightarrow \sigma \int_0^{\min(s,\vartheta)} B(t) \, dt.
$
Note that this limit process does not depend on $ \eta $.
\end{proof}

\begin{theorem}
\label{CPI1toI0}
Under the $I(1)$-to-$I(0)$ change-point model
(\ref{I1toI0}) we have
$$
  U_N(s) \Rightarrow \calU_{10,\vartheta}(s)
  = \frac{s^{-1} \zeta \int_0^s \biggl( \int_0^{\min(r,\vartheta)} B(t) \, dt \biggr)^2
          K( \zeta(r-s) ) \, dr }
         { \int_0^{\min(s,\vartheta)} B(t)^2 \, dt},
$$
as $ N \to \infty $, yielding
$$
  N^{-1} R_N \stackrel{d}{\to} \min\{ \kappa \le s \le 1 : \calU_{10,\vartheta}(s) > c \},
$$
as $ N \to \infty $.
\end{theorem}

\begin{proof} The theorem is proved using the same approach as in Theorem~\ref{CPI0toI1}.
We indicate the differences. First note that by Lemma~\ref{I1toI0Lemma}
\begin{eqnarray*}
  && \trunc{Ns}^{-3} \sum_{i=1}^{\trunc{Ns}} \biggl( \sum_{j=1}^i Y_j \biggr)^2
    K_h( t_i - t_{\trunc{Ns}} ) \\
   &=& \frac{N^3}{\trunc{Ns}^3} \frac{N}{h}
       \int_0^s \biggl( \int_0^r N^{-1/2} Y_{\trunc{Nt}} \, dt \biggr)^2
         K((\trunc{Nr}-\trunc{Ns})/h ) \, dr \\
   &\Rightarrow&
     s^{-1} \zeta \int_0^s \biggl( \sigma \int_0^{\min(r,\vartheta)} B(t) \, dt \biggr)^2
       K( \zeta(r-s) ) \, dr,
\end{eqnarray*}
as $ N \to \infty $. To handle the denominator of $ U_N $ observe that
$$
  \biggl\{ \int_0^s N^{-1} Y_{\trunc{Nt}}^2 \, dt : 0 \le s < \vartheta \biggr\}
  \Rightarrow
  \biggl\{ \sigma^2 \int_0^s B(t)^2 \, dt : 0 \le s < \vartheta \biggr\}.
$$
For $ s \ge \vartheta $ we obtain
\begin{eqnarray*}
  \trunc{Ns}^{-2} \sum_{j=1}^{\trunc{Ns}} Y_j^2
  &=& \frac{N^2}{\trunc{Ns}^2} \left( N^{-2} \sum_{j=1}^{\trunc{N\vartheta}-1} Y_j^2
   + N^{-2} \eta^2 \sum_{j=\trunc{N\vartheta}}^{\trunc{Ns}} u_j^2 \right)
  \Rightarrow s^{-2} \sigma^2 \int_0^\vartheta B(t)^2 \, dt
\end{eqnarray*}
yielding
$
\biggl\{ \trunc{Ns}^{-2} \sum_{j=1}^{\trunc{Ns}} Y_j^2 : s \in [0,1] \biggr\}
\Rightarrow
\biggl\{s^{-2} \sigma^2 \int_0^{\min(s,\vartheta)} B(t)^2 \, dt : s \in [0,1] \biggr\},
$
as $ N \to \infty $
\end{proof}

\begin{remark} Note that $ \calU_{10,\vartheta}$ does not depend on $ \eta $.
Hence, the detection procedure given by $ R_N $ is asymptotically robust w.r.t.
changes of the variance.
\end{remark}

\section{Simulations}
\label{Simulations}

We perform Monte Carlo simulations to investigate the actual finite sample
performances of the proposed monitoring procedure. We first consider the statistical
properties of the procedures,
if the time series is either $ I(0) $ or $ I(1) $. In a second step
we study the performance under change-point models. All simulations are based on
50,000 repetitions.

\subsection{Models without change-point (either $I(0)$ or $I(1)$)}

The first model we use for the simulations is as in Stock (1994a), an AR(1) process
with MA(1) errors,
$$
  Y_0 = 0, \ Y_{n} = \phi Y_{n-1} + e_n - \beta e_{n-1}, \ n = 1, \dots, N,
$$
where $ \phi $ and $ \beta $ are parameters and $ \{ e_n \} $ i.i.d.
$ N(0,1) $ innovations. The parameter values were chosen to be $
\phi = 1, 0.95, 0.9, 0.7 $ and $ \beta = -0.8, -0.5, 0, 0.5, 0.8 $.
We investigate the following quantities: Firstly, size and power of
the test which rejects $ H_0 $ if the monitoring procedure gives a
signal. Second, the average run length (ARL) defined as $ E(R_N) $
and $ E(\widetilde{R}_N) $, respectively, i.e., the average number
of observations until we get a signal. Additionally, we provide the
conditional ARL given that the procedure gives a signal (CARL)
defined as, e.g., $ E(R_N|R_N<N) $. That quantity informs us how
fast the procedure reacts if it reacts at all. We use a maximum
sample size of $ N = 250 $. The bandwidth was chosen as $ h = 50 $.
Simulated asymptotic critical values were used with $ \zeta = N/h =
5 $ to attain a nominal rejection probability of $ 5\% $. The
Gaussian and Epanechnikov kernels were investigated, which attach
smaller weights to past summands than to more current ones. We found
by simulations not reported here that the start of monitoring, $k$,
should be proportional to $h$, and $ k = 1.5 h $ yields a reasonable
rule of thumb for $ \zeta = 5 $.

Table~\ref{RWSimUnitRootPower2} presents our results for the proposed
procedure $ R_N $ to detect stationarity, using the
Epanechnikov kernel for weighting. Here $ H_0: $ $I(1)$-unit root is
given by $ \phi = 1 $.
The results are generally supportive of the theory developed in the paper. We do not
report the results for the Gaussian kernel, since they were quite similar.
The first three rows present the actual sizes for different values of $ \beta $.
It can be seen that there is only a slight size distortion, similar as for
the KPSS fixed-sample test. The remaining rows provide power estimates,
CARLs, and ARLs. Overall, it appears that the monitoring approach
provides a powerful method to detect quickly stationarity, as can be seen
from the CARL values in parentheses. In many cases stationarity can be detected
very early and it is not necessary to  wait until the time horizon $N$.

We next consider the properties of the procedure $ \widetilde{R}_N $ to detect
a unit root. Here $ H_0: $ $I(0)$-stationarity corresponds to $ | \phi | < 1 $
in our simulation model.
For the Newey-West estimator we have to choose the
lag truncation parameter $m$.
We considered the following choices of $ m $ as a function of the (current) sample size:
$ m3 = \trunc{ 0.75 n^{1/3} + 0.5 } $,
$ m4 = \trunc{ 4 (n/100)^{1/4} + 0.5 } $, and
$ m12 = \trunc{ 12 (n/100)^{1/4} + 0.5 } $ with
$ n = \trunc{N \kappa}, \dots, R_N \le N $
denoting the time point where the estimator has to be calculated.
The rules $ m4 $ and $ m12 $ have also been used by Kwiatkowski et al. (1992),
for $ m3 $ see Stock and Watson (2003, eq. 13.17).
For $ \phi = 0 $ and
$ \beta = -0.8, -0.5, 0, 0.5 $ we simulated the type I error
for all choices of $m$. As can be seen from the top rows of Table~\ref{RWSimI0I1},
the difference seems negligible. For the remaining cases given by
 $ \phi = 0.2, 0.6, 0.9, 1 $ we used $ m4 $.
The last three rows of the table provide the performance to detect the unit root given
by $ \phi = 1 $. Overall, the empirical rejection rates and ARL/CARL values indicate that
for moderate positive autocorrelation ($ 0 \le \phi \le 0.6 $) the procedure
has moderate size distortion. But, as expected, for $ \phi $ close to $1$
the procedure overreacts. The power is uniformly high for all values of $ \beta $
studied here.

\subsection{Change-point models}

We also investigated the performance of the detection methods $ R_N $ and
$ \widetilde{R}_N $ in change-point models. Of particular interest is to study
the influence of the bandwidth $h$ on the performance.
To evaluate the rule $ R_N $
(detection of stationarity), we used the following specification of the
change-point model given in the introduction,
$$
 Y_{n} = \phi_n Y_{n-1} + \epsilon_n \quad \mbox{where} \quad
 \left\{ \begin{array}{cc} \phi_n = 1, & n = 1, \dots, \trunc{N\vartheta}-1, \\
                        \phi_n = 0.5, & n = \trunc{N\vartheta}, \dots, N,
                        \end{array} \right.
$$
with $ N = 250 $. $ \{ \epsilon_t \} $ are i.i.d. $ N(0,1) $-innovations.
The change-point parameter
$ \vartheta $ is chosen as $ \vartheta = 0.1, 0.5, 0.75 $,
and the bandwidth as $ h = 125, 50, 25 $.

Table~\ref{RWSimChangePointI1I0} reports power,
the average delay, defined as $ E \max( R_N - \trunc{N\vartheta}, 0) $,
and the conditional average delay given the method provides a signal, defined
as  $ E( R_N | R_N < N ) - \trunc{N\vartheta} $, which informs us how quickly
the procedure reacts if it reacts at all. It can be seen that there is only a negligible
effect of the bandwidth $h$ on the average delay, but a remarkable
positive effect on the conditional average delay
and the statistical power to reject the unit root hypothesis.
Comparing $ h = 125 $ with $ h = 50 $ for $ \vartheta = 0.1 $ and $ \vartheta = 0.5 $
indicates that large bandwidths provide high overall power but the signal often comes late.
To detect the change early smaller bandwidths seem to be better. Comparing with $ h = 10 $
shows that CARL increases again. It seems, that for the setting studied
here values between $ 25 $ and $ 50 $ provide reasonable results.

To investigate the detection rule $ \widetilde{R}_N $ (detection of a unit root),
we used the same change-point model as above, but with $ \phi_n = 0.6 $
if $ n < \trunc{N \vartheta} $ and $ \phi_n = 1 $ if $ n \ge \trunc{N \vartheta} $.
This means, before the change the process is $ AR(1) $ with autoregressive parameter
$ 0.6 $, and after the change we are given a pure random walk.
The parameter $ \vartheta $ was chosen as above and $ h = 125, 50, 25, 10 $.
As can be seen from Table~\ref{RWSimChangePointI0I1}, the detection performance
is excellent in terms of power, average delay, and conditional average
delay. Results for $ h = 5 $ were almost identical to $ h = 10 $ and are
therefore omitted.
Overall, small bandwidths increase the power substantially and yield smaller
delays.

\subsection{An example}

Figure~\ref{Plot2} illustrates the detection performance of the proposed
procedure for a time series of length $ 250 $ which has a change-point.
The first 100 observations follow an AR(1) with
coefficient $ \phi = 0.8 $. After the change-point given by $ \vartheta = 0.4 $,
the series is a random walk ($\phi=1$). We applied the
procedure $ \widetilde{R}_N $ using the
Epanechnikov kernel, bandwidth $ h = 25 $, the lag selection rule $m4 $,
and an asymptotic $5\% $ control limit using $ \zeta = 10 $.
The change is detected at obs. $ 167 $.

\section{Software}

User-friendly and platform independent JAVA software implementing the proposed methods,
particularly providing asymptotic control limits, and example data sets can be downloaded from the author's webpage.

\section{Conclusions}

Monitoring rules to detect quickly stationarity and unit roots based
on a kernel-weighted process related to the KPSS statistics are
studied. Limiting distributions under various distributional
assumptions including local-to-unity and change-point models are
established. Simulations indicate that the procedures share the
moderate size distortion of the KPSS test, but due to its weighting
scheme controlled by a bandwidth parameter $h$ the reaction
performance is substantially improved. Both, changes from $I(0)$ to
$I(1)$ and changes from $ I(1) $ to $ I(0) $ can be detected in many
cases very early, if $h$ is chosen appropriately.

\section*{Acknowledgements}

The financial support of the DFG (Deutsche Forschungsgemeinschaft,
SFB 475, {\em Reduction of Complexity in Multivariate Data
Structures}) is gratefully acknowledged. I thank two anonymous
referees for constructive and helpful remarks which improved the
article, and Dipl.-Math. Sabine Teller for proof-reading a revised
version.

\newpage
%
\begin{center}
\begin{table}
\begin{tabular}{cccccc} \hline
   $ \phi $ &  \multicolumn{5}{c}{$\beta$} \\
           &  -0.8   &  -0.5    &  0       &  0.5    & 0.8   \\ \hline
$1$ &     $0.04$ & $0.04$ & $0.042$ & $0.051$ & $0.097$\\
    & $[171.9]$ & $[171.9]$ & $[171.7]$ & $[170.7]$ & $[165.2]$\\
$ 0.95 $ & $0.228$ & $0.23$ & $0.236$ & $0.285$ & $0.462$\\
         & $(101.2)$ & $(101.2)$ & $(100)$ & $(92.6)$ & $(70.7)$\\
         & $[158.2]$ & $[158]$ & $[157.3]$ & $[151.5]$ & $[126.9]$\\
$ 0.9 $  & $0.347$ & $0.352$ & $0.362$ & $0.443$ & $0.642$\\
         & $(92.2)$ & $(91.6)$ & $(90.2)$ & $(79.6)$ & $(52.2)$\\
         &$[146.3]$ & $[145.6]$ & $[144.3]$ & $[132.7]$ & $[96.2]$\\
$ 0.7 $  & $0.557$ & $0.557$ & $0.589$ & $0.717$ & $0.931$\\
         & $(69)$ & $(68.5)$ & $(64.5)$ & $(46.3)$ & $(22.4)$\\
         & $[116]$ & $[115.7]$ & $[109.9]$ & $[82.8]$ & $[33]$\\
\hline
 \end{tabular}
\caption{Detecting stationarity using $ R_N $ with $ \zeta = 5 $: Empirical size
respectively power, CARL given a signal (in parentheses), and ARL (in brackets) for
various values of $ \phi $ (AR parameter) and $ \beta $ (MA parameter).}
\label{RWSimUnitRootPower2}
\end{table}
\end{center}

\begin{center}
\begin{table}
\begin{tabular}{cccccc}  \hline
       & \multicolumn{4}{c}{$\beta$} \\
$\phi$ [lag rule] & $-0.8$ & $-0.5$ & $0$ & $0.5$  \\ \hline
$0$ [m3] & $0.036$ & $0.035$ & $0.023$ & $0.001$\\
$0$ [m4] & $0.033$ & $0.031$ & $0.022$ & $0.002$\\
$0$ [m12]&$0.016$ & $0.017$ & $0.017$ & $0.005$\\
$0.2$    & $0.039$ & $0.038$ & $0.03$ & $0.005$\\
         & $[173.2]$ & $[173.3]$ & $[173.7]$ & $[174.8]$\\
$0.6$    & $0.082$ & $0.083$ & $0.074$ & $0.039$\\
         & $[170.4]$ & $[170.3]$ & $[171]$ & $[173.2]$\\
$0.9$    & $0.396$ & $0.399$ & $0.391$ & $0.358$\\
         & $[140]$ & $[139.7]$ & $[140.6]$ & $[144.9]$\\
$1$      & $0.952$ & $0.953$ & $0.955$ & $0.951$\\
         & $(51.3)$ & $(51.3)$ & $(51.3)$ & $(51)$\\
         & $[57.2]$ & $[57.1]$ & $[56.9]$ & $[57.1]$\\
\hline
\end{tabular}
\caption{Detecting unit roots using $\widetilde{R}_N$ with $\zeta=5$: Empirical size
respectively power, CARL given a signal (in parentheses), and ARL (in brackets),
 for various values of $ \phi $ (AR parameter) and $ \beta $ (MA parameter).}
\label{RWSimI0I1}
\end{table}
\end{center}

\begin{center}
\begin{table}
\begin{tabular}{cccc} \hline
            &  \multicolumn{3}{c}{$ \vartheta $} \\
            &  $ 0.1 $  & $ 0.5 $ & $ 0.75 $ \\ \hline
$ h = 125 $ & $0.385$ & $0.072$ & $0.066$\\
            & $(133.2)$ & $(54.6)$ & $(41.7)$\\
            & $[189.7]$ & $[119.9]$ & $[60.3]$\\
$ h = 50$   & $0.293$ & $0.051$ & $0.055$\\
            & $(112.6)$ & $(39.4)$ & $(21.7)$\\
            & $[192]$ & $[120.2]$ & $[60.1]$\\
$ h = 25 $  & $0.247$ & $0.041$ & $0.046$\\
            & $(113.5)$ & $(37.4)$ & $(15.4)$\\
            & $[197.4]$ & $[120.9]$ & $[60.4]$\\
$ h = 10 $  & $0.208$ & $0.031$ & $0.037$\\
            & $(117.3)$ & $(40)$ & $(13.8)$\\
            & $[202.6]$ & $[121.9]$ & $[60.8]$\\
\hline
 \end{tabular}
\caption{Change from $I(1)$ to $I(0)$: Empirical rejection rates, conditional average
delay (in parentheses), and unconditional average delay (in brackets).}
\label{RWSimChangePointI1I0}
\end{table}
\end{center}

\begin{center}
\begin{table}
\begin{tabular}{cccc} \hline
            &  \multicolumn{3}{c}{$ \vartheta $} \\
            & $ 0.1 $ & $ 0.5 $  & $ 0.75 $ \\ \hline
$ h = 125 $ & $0.94$ & $0.744$ & $0.155$\\
            & $(123.7)$ & $(87.9)$ & $(49.7)$\\
            & $[129.8]$ & $[97.4]$ & $[58.5]$\\
$ h = 50 $  & $0.966$ & $0.854$ & $0.468$\\
            & $(103.2)$ & $(74)$ & $(49.4)$\\
            & $[107.4]$ & $[80.6]$ & $[53.7]$\\
$ h = 25 $  & $0.972$ & $0.895$ & $0.628$\\
            & $(94.8)$ & $(65)$ & $(44.3)$\\
            & $[98.5]$ & $[70.4]$ & $[48.4]$\\
$ h = 10 $  & $0.974$ & $0.913$ & $0.702$\\
            & $(89.8)$ & $(58.2)$ & $(39.4)$\\
            & $[93.4]$ & $[63.1]$ & $[43.8]$\\
\hline
\end{tabular}
\caption{Change from $I(0)$ to $I(1)$: Empirical rejection rates, conditional average
delay (in parentheses), and unconditional average delay (in brackets).}
\label{RWSimChangePointI0I1}
\end{table}
\end{center}

\begin{figure}
\includegraphics[height = 10cm]{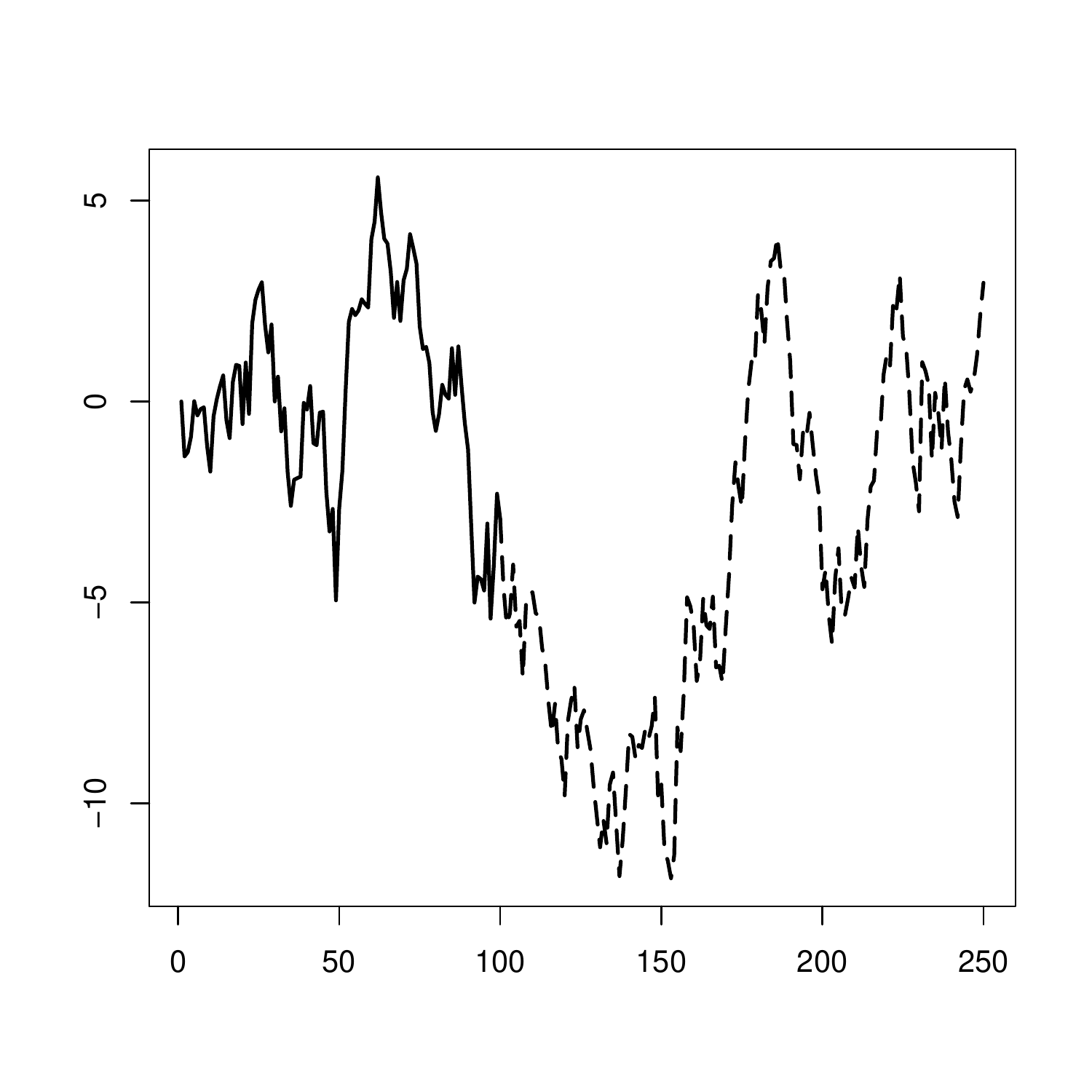}
\caption{A time series with a change-point at obs. 100 where the AR coefficient changes
from $0.8$ to $1$. The change is detected at obs. 161.} \label{Plot2}
\end{figure}

\end{document}